\documentclass[a4paper,12pt]{article}

\usepackage{amsthm}   \usepackage{amsfonts}     \usepackage{amssymb}
\usepackage{amsmath}  \usepackage{mathrsfs}     \usepackage{mathtools}
\usepackage{url}       \usepackage{fancyhdr}

\usepackage[all]{xy}    \usepackage{amscd}   \usepackage{epsfig}
\usepackage{tikz}

\usepackage{array}   \usepackage{multirow}  \usepackage{booktabs,caption,fixltx2e}
\usepackage[flushleft]{threeparttable}

\usepackage{cases} 
\usepackage{ulem}  
\usepackage{xcolor}  

\usepackage[T1]{fontenc} 
\usepackage[utf8]{inputenc}
\usepackage{authblk}  
\usepackage{lineno}

\newtheorem{theorem}{Theorem}
\newtheorem{lemma}[theorem]{Lemma}
\newtheorem{proposition}[theorem]{Proposition}
\newtheorem{corollary}[theorem]{Corollary}

\theoremstyle{definition}
\newtheorem{definition}[theorem]{Definition}
\newtheorem{example}{Example}

\newtheorem{remark}{Remark}

\newcommand{\Z}{{\mathbb{Z}}}
\newcommand{\A}{\mathcal{A}}
\newcommand{\C}{\mathcal{C}}
\newcommand{\D}{\mathcal{D}}

\newcommand{\M}{\mathcal{M}}

\newcommand{\Aut}{{\mathrm{Aut}}}
\newcommand{\Inn}{{\mathrm{Inn}}}
\newcommand{\Out}{{\mathrm{Out}}}
\newcommand{\Mon}{{\mathrm{Mon}}}
\newcommand{\Gal}{{\mathrm{Gal}}}
\newcommand{\normal}{\trianglelefteq}

\newcommand{\la}{\langle}
\newcommand{\ra}{\rangle}
\newcommand{\Sym}{{\mathrm{Sym}}\,}
\newcommand{\Alt}{\mathrm{Alt}}

\newcommand{\lcm}{{\mathrm{lcm}}\,}

\begin{document}

\title{Nilpotent dessins: Decomposition theorem and classification of the abelian dessins}
\author[1,2,3]{Kan Hu\thanks{kanhu@savbb.sk}}
\author[1,2]{Roman Nedela\thanks{nedela@savbb.sk}}
\author[1,3]{Na-Er Wang\thanks{naerwang@savbb.sk}}
\affil[1]{Faculty of Natural Sciences, Matej Bel University, Tajovsk\'eho 40, 974 01, Bansk\'a Bystrica, Slovak Republic}
\affil[2]{Institute of Mathematics and Computer Science, Slovak Academy of Sciences, Bansk\'a Bystrica, Slovak Republic}
\affil[3]{School of Mathematics, Physics and Information Science, Zhejiang Ocean University, Zhoushan, Zhejiang 316000, People's Republic of China}
\maketitle
\begin{abstract}
A map is a 2-cell decomposition of an orientable closed surface.
A dessin is a bipartite map with a fixed colouring of vertices. A dessin is regular if its group of colour- and orientation-preserving automorphisms acts transitively on the edges, and a regular dessin is symmetric if it admits an additional external symmetry transposing the vertex colours. Regular  dessins with nilpotent automorphism groups are investigated. We show that each such dessin is a parallel product of regular dessins whose automorphism groups are the Sylow subgroups. Regular and symmetric dessins with abelian automorphism groups are classified and enumerated. 
\\[2mm]
\end{abstract}


\section{Introduction}
A \textit{map} $\mathcal{M}$ is 2-cell decomposition of a closed surface. Maps are often described as 2-cell embeddings of connected graphs into surfaces. Graphs considered in this paper may have multiple edges. A map is \textit{orientable} if its supporting surface is orientable, otherwise, it is \textit{non-orientable}. All maps considered in this paper are  orientable, unless otherwise stated.  A map is {\it bipartite} if its underlying graph is bipartite, that is, the vertices can be coloured in white and black colours such that no vertices of the same colour are adjacent. A \textit{dessin} $\D$ is a bipartite map with a fixed colouring of vertices.

  An (orientation-preserving)  \textit{automorphism} of a map $\M$ is an automorphism of the underlying graph, defined as an incidence-preserving permutation of the darts (arcs), which extends to an orientation-preserving self-homeomorphism of the supporting surface. The set of automorphisms of a map $\M$ forms the automorphism group $\Aut^+(\M)$ of $\M$ under composition. It is well-known that $\Aut^+(\M)$ acts semi-regularly on the darts of $\M$. If this action is transitive, and hence regular, then the map $\M$ itself is called \textit{regular}.
The \textit{automorphism group} $\Aut(\D)$ of a dessin  $\D=\D(\M)$ associated with a bipartite map $\M$ is formed by the automorphisms of $\M$ preserving the colours.  A dessin $\D$ is called \textit{regular}, if $\Aut(\D)$ is regular on the set of edges.

Grothendieck discovered that there
is a close relation between Riemann surfaces and dessins~\cite{Grothendieck1997}.  By the Bely\v{\i}'s Theorem a compact Riemann surface $\mathcal{S}$, regarded as a complex projective algebraic curve, is defined over the field $\mathbb{\bar Q}$ of algebraic numbers is equivalent to the existence of a non-constant meromorphic function on $\mathcal{S}$ with at most three critical values~\cite{Be1, Be2}. The surface is called a Bely\v{\i} surface and the function is called a Bely\v{\i} function. Each Bely\v{\i} function determines in a canonical way a dessin on $\mathcal{S}$, and vice versa~\cite{JS1996}. The absolute Galois group $\Gal(\mathbb{\bar Q}/\mathbb{Q})$  has a natural action on the Bely\v{\i} surfaces, and hence on the dessins. Therefore, dessins present a combinatorial approach to  the absolute Galois group $\Gal(\mathbb{\bar Q}/\mathbb{Q})$ through its action on particular classes of dessins.

Between dessins regular dessins are of particular interest, because the absolute Galois group is faithful on this class of dessins \cite{GZ2013}, and each dessin is a quotient of some regular dessin by a group of automorphisms. One of the most important problems in this field is the classification of regular dessins and the associated algebraic curves. This has been investigated by posing certain conditions on the underlying group of automorphisms, on the underlying graphs, or on the supporting surfaces; see \cite{BJ2001, Conder2012, DJKNS2007, DJKNS2010, Jones2009, Jones2010-2, Jones2010,Jones2013,JNS2008, JNS2007, CJSW2009,Hidalgo2013,Jones2012,JSW2007,JSW2010} for details.

Our main concern here is the first direction. Employing Hall's method given in \cite{Hall1936}, Downs and Jones developed a general approach to enumerate regular objects (including regular dessins) with a given automorphism group~\cite{DJ1987}. The reader is referred to \cite{Hall1936,DJ1987} and \cite{DJ2013} for its applications to the simple groups $PSL_2(p^e)$ and $Sz(2^e)$, and to \cite{Jones2013} for an excellent survey on this topic. In this paper, inspired by some classification results of nilpotent regular maps in~\cite{BDLNS,MNS2012}, we investigate regular dessins with nilpotent automorphism groups. We present a decomposition theorem reducing the classification problem to the classification of regular $p$-dessins -- regular dessins whose automorphism groups are $p$-groups, see Decomposition Theorem~\ref{REDUCTION}.

  The simplest nilpotent regular dessins are the abelian ones.
As noted by Jones and Wolfarth, a nice feature of the abelian dessins
is that they are all quotients of the Fermat dessins.
 The curves related to particular classes of abelian dessins were extensively investigated  in  literature, see for instance \cite{CaSj,GoRu}. They include some important and popular families of curves such as curves of Fermat and Lefschetz type.
Although there are many papers dealing with abelian and cyclic dessins and
with the counterparts in the theory of algebraic curves, it is hard to give a universal reference, since many facts are known between experts as a folklore.

   In this paper, as an application of the Decomposition Theorem~\ref{REDUCTION}, abelian regular dessins are classified and enumerated, see Theorem~\ref{CYCLIC}, \ref{ACLASS}, \ref{ANUMBER}.  Their basic combinatorial properties are determined, which extend the results of Hidalgo in \cite{Hidalgo2013}. It will serve as a basis of induction for the classification of nilpotent regular dessins of higher nilpotency classes, which we want to study as a sequel to this work.

\section{Regular dessins}
A \textit{dessin} $\D$ is an embedding of a connected bipartite graph $\mathcal{B}$ with a fixed (black and white) colouring into an orientable compact surface $\mathcal{S}$ without boundary, dividing $\mathcal{S}$ into simply connected \textit{faces}. Following the global orientation of the supporting surface $\mathcal{S}$, one obtains two permutations $\rho$ and $\lambda$ which successively permute the edges around the black and white vertices. Due to the connectivity, the group generated by $\rho$ and $\lambda$ acts transitively on the edges of $\D$. Conversely, each two-generated transitive permutation group $\la \rho,\lambda\ra\leq \Sym(\Phi)$ where $\Phi$ is a finite nonempty set determines a dessin: We identify the elements of $\Phi$ with the edges, and the cycles of $\rho$ and $\lambda$ with the black and white vertices with the incidence given by containment. In this way we obtain a coloured bipartite graph $\mathcal{B}$. The successive powers of $\rho$ and $\lambda$ give the rotation of edges around each vertex, and these local orientations determine an embedding of $\mathcal{B}$ in an oriented surface.

Therefore, we have an alternative and equivalent way to describe dessins combinatorially. More specifically, a \textit{combinatorial dessin} is a triple $(\Phi;\rho,\lambda)$ consisting of a nonempty finite set $\Phi$ and two permutations $\rho$ and $\lambda$ on $\Phi$ such that the group  $\Mon(\D)=\langle\rho,\lambda\rangle$, called the \textit{monodromy group} of $\D$, is transitive on $\Phi$. An (orientation- and colour-preserving) \textit{homomorphism} (or \textit{covering}) $\D_1\to\D_2$ between two dessins $\D_i=(\Phi_i;\rho_i,\lambda_i)$ $(i=1,2)$  is a mapping $\phi:\Phi_1\to \Phi_2$ such that $\phi\rho_1=\rho_2\phi$ and $\phi\lambda_1=\lambda_2\phi$. A bijective homomorphism between two dessins is called an \textit{isomorphism}, and a self-homomorphism of a dessin $\D$ is called an \textit{automorphism} of $\D$. The set of automorphisms of $\D$ forms the automorphism group $\Aut(\D)$ of $\D$ under composition. It follows that the automorphism group $\Aut(\D)$ is the centraliser of $\Mon(\D)$ in $\Sym(\Phi)$. Due to the transitivity of $\Mon(\D)$, $\Aut(\D)$ acts semi-regularly on $\Phi$. In  case that $\Mon(\D)$ is regular on $\Phi$, or equivalently, $\Aut(\D)$ is  regular on $\Phi$, the dessin $\D$ itself is called \textit{regular}.

A dessin $\D=(\Phi;\rho,\lambda)$ is said to be of \textit{type} $(l,m,n)$ where
\[
l=o(\rho),\quad m=o(\lambda), \quad n=o(\rho\lambda).
\]
Each dessin $\D=(\Phi;\rho,\lambda)$ can be regarded as a transitive permutation representation $\phi:\Delta\to\Sym(\Phi)$ of the triangle group
\[
\Delta=\la X,Y\mid -\ra=\la X,Y,Z\mid XYZ=1\ra
\]
given by the assignment $X\mapsto \rho, Y\mapsto \lambda$. In this way, every dessin $\D$ corresponds to a subgroup $N$ of finite index in $\Delta$ given by the stabilizer of an element of $\Phi$. It is called the \textit{dessin subgroup} associated with $\D$, and it is uniquely determined up to conjugacy. Regular dessins $\D$ with $\Aut(\D)\cong G$ are therefore in one-to-one correspondence with the normal subgroups $N$ of $\Delta$ such that $\Delta/N\cong G$.

An automorphism $\sigma$ of $\Delta$ sends $N$ to $N^{\sigma}$ and hence transforms the associated dessin $\D$ to $\D^{\sigma}$. In particular, if $\sigma\in\Inn(\Delta)$ then $\D\cong\D^{\sigma}$. It follows that the outer automorphism group $\Out(\Delta)=\Aut(\Delta)/\Inn(\Delta)$ of $\Delta$ induces a (faithful) action, called the group $\Omega$ of \textit{dessin operations}, on the isomorphism classes of dessins~\cite{JT1983, James1988, JP2010}.
One nice feature of dessin operations is that they preserve the monodromy group (and hence the automorphism group) of the dessin. The most prominent dessin operations are the \textit{duality operations} and the \textit{generalised Wilson's operation} $H_{i,j}$. The former transform a dessin to a dessin on the same surface by permuting the black vertices, the white vertices and the faces,   and the latter transforms a dessin $\D=(\Phi;\rho,\lambda)$ of type $(l,m,n)$ to a dessin $H_{i,j}(\D)=(\Phi;\rho^i,\lambda^j)$ of type $(l,m,n')$ where $\gcd(i,l)=\gcd(j,m)=1$ and $n'$ is a possible different number from $n$~\cite{JSW2010}. In particular, the operation $H_j=H_{j,j}$ is called the \textit{Wilson's operation}~\cite{Wilson1979}.

In a regular dessin $\D=(\Phi;\rho,\lambda)$, $\Mon(\D)\cong\Aut(\D)$.
In this case we can identify the edges of $\D$ with the elements of $G=\Aut(\D)$. The groups $\Mon(\D)$ and $\Aut(\D)$ are then identified with the left and right regular representations of $G$. The two generators of $G$ can be chosen as the automorphisms $x$ and $y$ which stabilise, respectively,  two adjacent
vertices. In this way, each regular dessin is identified with an \textit{algebraic dessin} -- a triple $(G,x,y)$. It follows that two dessins $(G_i,x_i,y_i)$ $(i=1,2)$ are isomorphic if the assignment $x_1\mapsto x_2, y_1\mapsto y_2$ extends to a group isomorphism from $G_1$ onto $G_2$. In particular, if $G=G_1=G_2$, then two generating pairs $(x_i,y_i)$ $(i=1,2)$ of $G$ give rise to isomorphic dessins with the same underlying group $G$ if and only if $(x_1, y_1)$ is \textit{equivalent} to $(x_2,y_2)$, that is, the assignment $x_1\mapsto x_2, y_1\mapsto y_2$ extends to an automorphism of $G$. Therefore, by the semi-regularity of the action of $\Aut(G)$ on the generating pairs of $G$, we have
\begin{theorem}\label{CORR}
Let $G$ be a two-generated group, then the isomorphism classes of regular dessins $\D$ with $\Aut(\D)\cong G$ are in one-to-one correspondence with the orbits of $\Aut(G)$ on the set $\mathcal{P}(G)$ of generating pairs of $G$. In particular, the number of non-isomorphic regular dessins $\D$ with $\Aut(\D)\cong G$ is equal to $\frac{|\mathcal{P}(G)|}{|\Aut(G)|}.$
\end{theorem}

The type, the Euler characteristic $\chi$ and the genus $g$ of an algebraic
dessin can be expressed in group-theoretical terms as follows.

\begin{proposition}
\label{EPformula}
Let $\D=(G, x,y)$ be an algebraic dessin. Then it has type $(o(x),o(y), o(xy))$,
and the genus $g$ and Euler characteristic $\chi$ of $\D$ are determined by the Euler-Poincar\'e formula:
\[
2-2g=\chi=|G|\left(\frac{1}{o(x)}+\frac{1}{o(y)}+\frac{1}{o(xy)}-1\right),
\]
where $o(x)$, $o(y)$ and $o(xy)$ are the orders of the respective elements in $G$.
\end{proposition}

Regular dessins possessing additional external symmetries have attracted particular attention. More specifically, let $\D$ be a regular dessin and $N$ the associated dessin subgroup, let $\sigma\in\Aut(\Delta)$. The dessin $\D$ is said to possess an \textit{external symmetry} $\sigma$ if $\D\cong\D^{\sigma}$, that is, $N$ is a $\sigma$-invariant subgroup of $\Delta$. In particular, a regular dessin is \textit{symmetric} if the automorphism $\tau:X\mapsto Y, Y\mapsto X$ of $\Delta$ fixes $\D$, \textit{reflexible} if the automorphism $\iota:X\mapsto X^{-1}, Y\mapsto Y^{-1}$ of $\Delta$ fixes $\D$, and \textit{totally symmetric} if $\D$ is fixed by any automorphism of $\Delta$.
Regular dessins which are not reflexible are called \textit{chiral}.
Topologically speaking, symmetric dessins have a symmetry transposing the vertex colours, and reflexible dessins have a symmetry reversing the orientation of the supporting surfaces.

In the language of algebraic dessins, if a regular dessin $\D=(G,x,y)$ possesses an external symmetry $\sigma\in\Aut(\Delta)$, then the associated (normal) dessin subgroup $N$ of $\Delta$ is $\sigma$-invariant, that is, $\sigma(N)=N$. Hence the assignment $x\mapsto  \sigma(x), y\mapsto \sigma(y)$ extends to an automorphism of $G$ where by abuse of notation we denote by $\sigma$ the automorphism of $G=\Delta/N$ induced by the automorphism $\sigma$ of $\Delta$. In this case, the generating pair $(x,y)$ of $G$ will be said to possess \textit{$\sigma$-property}. Specifically, if the assignment $\tau:x\mapsto y, y\mapsto x$ extends to an automorphism of $G$, then $(x,y)$ will be called a \textit{transpositional} generating pair $(x,y)$ of $G$, and if the assignment $\iota:x\mapsto x^{-1}, y\mapsto y^{-1}$ extends to an automorphism of $G$, then $(x,y)$ will be called an \textit{invertible} generating pair $(x,y)$ of $G$. So symmetric dessins $(G,x,y)$ correspond to transpositional generating pairs $(x,y)$ of $G$, and reflexible dessins $(G,x,y)$ correspond to invertible generating pairs $(x,y)$ of $G$.

Note that in general the underlying graph of a regular dessin $\D=(G,x,y)$ may have multiple edges.  The appearance of multiple edges reflects an unfaithful action of $G=\Aut(\D)$ on the (black and white) vertices, with a nontrivial core $K=\la x\ra\cap \la y\ra$. In this context, the associated (quotient) simple regular dessin $\bar\D=(G/K,xK,yK)$ will be called the \textit{shadow dessin} of $\D$.

In what follows, we investigate the behaviour of regular dessin possessing external symmetries with respect to coverings.  Recall that a normal subgroup $H\normal G$ is called $\sigma$-invariant if $\sigma(H)=H$ where $\sigma\in\Aut(G)$.

\begin{proposition}\label{QUOT}
Let $\D=(G,x,y)$ be a regular dessin possessing an external symmetry $\sigma$. If $H$ is a $\sigma$-invariant subgroup of $G$, then the quotient dessin $\bar\D=(G/H, xH, yH)$ also possesses the external symmetry $\sigma$.
\end{proposition}
\begin{proof}
Let $N$ and $M$ be the dessin subgroups associated with $\D$ and $\bar\D$. Then $\D$ covers $\bar\D$ with an $\sigma$-invariant subgroup $H\cong M/N$ of covering transformations in $ G\cong\Delta/N$. Since $\D$ possesses the external symmetry $\sigma$, $N$ is $\sigma$-invariant in $\Delta$. It follows that $M$ is $\sigma$-invariant in $\Delta$ as well. Therefore, $\bar\D$ possesses the external symmetry $\sigma$.
\end{proof}

\begin{proposition}\label{TPROPERTY}
Let $\D=(G,x,y)$ be a regular dessin of type $(l,m,n)$, and let $K=\la x\ra\cap\la y\ra$ be of order $k$ where $k\geq 1$. Then the generators $x$ and $y$ satisfy the relation $x^{l/k}=y^{em/k}$ where $e$ is coprime to $k$. In particular, if $\D$ is symmetric, then $l=m$ and $e^2\equiv1\pmod{k}.$
\end{proposition}
\begin{proof}Since $K=\la x^{l/k}\ra=\la y^{m/k}\ra$, we have $x^{l/k}=(y^{m/k})^e$ for some integer $e$ coprime to $k$. In particular, if $\D$ is symmetric, then $l=m$, and the above relation is reduced to $x^{l/k}=y^{el/k}$. Applying the automorphism $\tau$ of $G$ transposing $x$ and $y$, we get $y^{l/k}=x^{el/k}$. It follows that $x^{l/k}=y^{el/k}=x^{e^2l/k}.$ Therefore, $e^2\equiv 1\pmod{k}$, as required.\end{proof}

\begin{corollary}
The shadow dessin of a symmetric dessin is symmetric, and the shadow dessin of a reflexible regular dessin is reflexible regular.\end{corollary}
\begin{proof}Let $\D=(G,x,y)$ and $K=\la x\ra\cap\la y\ra$. Observe that $K$ is a cyclic central subgroup of $G$ generated by some power of $x$, we have $K\normal G$.

Assume that $\D$ is symmetric with $\tau\in\Aut(G)$ transposing $x$ and $y$. If $x^i\in K$, then $\tau(x^i)=y^i\in \la y\ra$. Since $o(x)=o(y)$, the subgroups $\la x^i\ra$ and $\la y^i\ra$ of the cyclic group $\la y\ra$, being of the same order, are identical. It follows that
$y^i\in\la x\ra$. Hence, $\tau(x^i)=y^i\in K$.  Therefore, $K$ is $\tau$-invariant. By Proposition~\ref{QUOT}, the shadow dessin $\D/K$ is also symmetric.

Now assume that $\D$ is reflexible with $\iota\in\Aut(G)$ inverting $x$ and $y$. Since $\la x^i\ra=\la x^{-i}\ra$, $K$ is $\iota$-invariant. By Proposition~\ref{QUOT}, the shadow dessin $\D/K$ is also reflexible.
\end{proof}

In what follows, we study a few examples of regular dessins, with  attention on those with external symmtries.
\begin{example} \label{CUBE} Consider the alternating group  $G=\Alt_4$. It can be easily verified that if $(x,y)$ is a generating pair of $G$, then either one of the generators is a permutation of cycle type $(2,2)$ and the other is of cycle type $(1,3)$, or both are of cycle type $(1,3)$. For instance, the pairs
\[
\begin{array}{ll}
\mathbf{P}_1=\big((12)(34), (123)\big), &\mathbf{P}_2=\big((123), (12)(34)\big),\\
\mathbf{P}_3=\big((123), (124)\big), &\mathbf{P}_4=\big((132), (124)\big)
\end{array}
\]
all generate $G$, and they are pairwise non-equivalent. More specifically, $G$ has precisely $96$ distinct generating pairs. Since $\Aut(G)\cong \mathrm{Sym}_4$, by Theorem~\ref{CORR},  there are precisely $96/24=4$ non-isomorphic dessins $\D_i$ $(i=1,2,3,4)$ given by the generating pairs $\mathbf{P}_i$ above. The dessins $\D_1$ and $\D_2$ are reflexible but asymmetric. $\D_1$ can be viewed as the subdivided tetrahedral map on the sphere, with its vertices coloured by black and the added vertices at the edge centres coloured by white. $\D_2$ is the dual of $\D_1$ by swapping the vertex colours. On the other hand, the dessins $\D_3$ and $\D_4$ are symmetric and reflexible. They correspond to the regular embeddings of the cube into the sphere and into the torus, respectively.
\end{example}

\begin{example}Let $G=\la g,h\ra$ be a metacyclic $2$-group of order $64$ defined by the presentation
\begin{align}\label{META64}
\la g,h\mid g^8=h^8=1,h^g=h^5\ra.
\end{align}
Since $\la g\ra\cap\la h\ra=1$, each element of $G$ can be uniquely written as form $g^ih^j$ where $0\leq i,j\leq 7$. Let $x=g^ih^j$ and $y=g^kh^l$. By the Burnside's Basis Theorem~\cite[Theorem~3.15]{Huppert1967}, $G=\la x,y\ra$ if and only if the matrix
$A(i,j,k,l)=\begin{pmatrix}
i&j\\k&l
\end{pmatrix}
$
is invertible in $\Z_2$, that is,
\begin{align}\label{BURN}
il-jk\not\equiv0\pmod{2}.\end{align} Note that each invertible matrix $A(i,j,k,l)$ where $i,j,k,l\in\Z_2$ lifts to $4^4$ invertible matrices $A(i,j,k,l)$ such that $i,j,k,l\in\Z_8$. Since $|GL(2,2)|=6$, the group $G$ has precisely $6\cdot 4^4=3\cdot 2^9$ generating pairs. On the other hand, for a generating pair $(x,y)$ of $G$ given above, the assignment $g\mapsto x, h\mapsto y$ extends to an automorphism of $G$ if and only if
\[
o(x)=o(y)=8\quad\text{and}\quad y^x=y^5.
\]
One can easily deduce from the presentation \eqref{META64} that $[h,g]=h^4\in Z(G)$ and $[h,g]^2=1$. By \eqref{BURN}, we have $x^4=g^{4i}h^{4j}[h,g]^{ij{4\choose 2}}=g^{4i}h^{4j}\neq1$ and $x^8=1$. Hence, $o(x)=8$. Similarly, $o(y)=8$. Moreover, we have
\begin{align*}
y^x=&(g^kh^l)^{g^ih^j}=h^{-j}g^kh^{j+ 5^i\cdot l}=g^kh^{(1-5^k)j+ 5^i\cdot l},\\
y^5=&(g^kh^l)^5=g^{5k}h^{5l}[h,g]^{10lk}=g^{5k}h^{5l}.
\end{align*}
It follows that $y^x=y^5$ if and only if $5k\equiv k\pmod{8}$ and $(1-5^k)j+ 5^i\cdot l\equiv 5l\pmod{8}$. This is equivalent to that $i$ and $l$ are both odd, and $k$ is even. Therefore $|\Aut(G)|=2^9$. Consequently, by Theorem~\ref{CORR}, there are precisely 3 non-isomorphic regular dessins $\D_i$ $(i=1,2,3)$ with $\Aut(\D_i)\cong G$, corresponding to three non-equivalent generating pairs $(g,h)$, $(h,g)$ and $(g,gh)$ of $G$. The first two dessins are simple, reflexible but not symmetric, each being the dual of the other by swapping the vertex colours. They correspond to the edge-transitive (irregular) embeddings of $K_{8,8}$. The third dessin is simple, reflexible and symmetric, corresponding to the regular embedding of $K_{8,8}$~\cite{DJKNS2007}.
\end{example}

\begin{example}\label{QUAT}
The quaternion group $Q_8$ contains six elements of order $4$, and a central involution. For each element $x$ of order $4$, there are precisely $4$ elements $y$ which can be paired with $x$ such that $Q_8=\la x,y\ra$. Therefore, $Q_8$ has precisely $6\times 4=24$ generating pairs. Since $\Aut(Q_8)\cong\mathrm{Sym}(4)$, by Theorem~\ref{CORR}, there is only one regular dessin $\D$ such that $\Aut(\D)\cong Q_8$. It is totally symmetric of type $(4,4,4)$ and genus $2$, corresponding to the $8$-gonal regular embedding of $K_{2,2}^{(2)}$ into the double torus.
\end{example}

\begin{example}\label{P3}
Let $p>2$ be a prime, and let $G$ be the (unique) non-abelian non-metacyclic $p$-group of order $p^3$ given by the presentation
\begin{align}\label{PCUBE}
G=\la x,y\mid x^p=y^p=z^p=[z,x]=[z,y]=1,z:=[x,y]\ra.
\end{align}
Clearly, $\Phi(G)=G'=\la z\ra$ and ${\rm exp}(G)=p$. Since $[x,z]=[y,z]=1$, each element of $G$ can be written as form $x^iy^jz^k$ where $i,j,k\in\Z_p$. Let $x_1=x^iy^jz^k$ and $y_1=x^ry^sz^t$, then by the Burnside Basis Theorem~\cite[Theorem~3.15]{Huppert1967}, $G=\la x_1,y_1\ra$ if and only if $G/\Phi(G)=\la \bar x^i\bar y^j, \bar x^r\bar y^s\ra$. Since $G/\Phi(G)\cong\Z_p\oplus\Z_p$, this is equivalent to that the matrix
$\begin{pmatrix}
i&j\\r&s
\end{pmatrix}
$ is invertible in $\Z_p$, that is, $is-jr\not\equiv0\pmod{p}$. Therefore, $G$ has precisely $p^3(p-1)^2(p+1)$ generating pairs. Moreover, it can be easily verified that for each such generating pair $(x_1,y_1)$ of $G$, the assignment $x\mapsto x_1, y\mapsto y_1$ is an automorphism of $G$. Hence, $|\Aut(G)|=p^3(p-1)^2(p+1)$ as well. By Theorem~\ref{CORR}, there is a unique regular dessin $\D$ such that $\Aut(\D)\cong G$, each being totally symmetric of type $(p,p,p)$. For instance, if $p=3$, then the  dessin corresponds to the regular embedding of the Pappus graph -- a cubic bipartite graph with 18 vertices and 27 edges determined as the Levi graph of the Pappus configuration \cite{Coxeter1950} -- into the torus, see Figure~\ref{PAPPUS}.
\end{example}

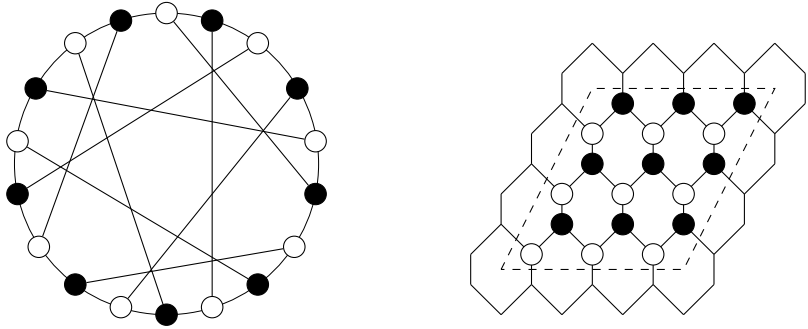
\begin{figure}[h!]
\begin{center}
\begin{tikzpicture}[scale=0.2, inner sep=1mm]

\draw [very thin] (-20,10) circle (10);
\node (a) at (-20,0)[shape=circle, draw, fill=black] {};
\node (a1) at (-17,0.5)[shape=circle, draw, fill=white] {};
\node (a2) at (-14,2)[shape=circle, draw, fill=black] {};
\node (a3) at (-11.6,4.5)[shape=circle, draw, fill=white] {};
\node (a4) at (-10.2,8)[shape=circle, draw, fill=black] {};
\node (a5) at (-10.2,11.5)[shape=circle, draw, fill=white] {};
\node (a6) at (-11.4,15)[shape=circle, draw, fill=black] {};
\node (a7) at (-14,18)[shape=circle, draw, fill=white] {};
\node (a8) at (-17,19.5)[shape=circle, draw, fill=black] {};
\node (b) at (-20,20)[shape=circle, draw, fill=white] {};
\node (b1) at (-23,0.5)[shape=circle, draw, fill=white] {};
\node (b2) at (-26,2)[shape=circle, draw, fill=black] {};
\node (b3) at (-28.4,4.5)[shape=circle, draw, fill=white] {};
\node (b4) at (-29.8,8)[shape=circle, draw, fill=black] {};
\node (b5) at (-29.8,11.5)[shape=circle, draw, fill=white] {};
\node (b6) at (-28.6,15)[shape=circle, draw, fill=black] {};
\node (b7) at (-26,18)[shape=circle, draw, fill=white] {};
\node (b8) at (-23,19.5)[shape=circle, draw, fill=black] {};
\draw (a) to (b7);
\draw (a4) to (b);
\draw (a1) to (a8);
\draw (a2) to (b5);
\draw (a3) to (b2);
\draw (a5) to (b6);
\draw (a6) to (b1);
\draw (a7) to (b4);
\draw (b3) to (b8);

\draw (0,2)--(2,0);
\draw (4,2)--(6,0);
\draw (8,2)--(10,0);
\draw (12,2)--(14,0);

\draw (4,2)--(2,0);
\draw (8,2)--(6,0);
\draw (12,2)--(10,0);
\draw (16,2)--(14,0);

\draw (0,2)--(0,4);
\draw (4,2)--(4,4);
\draw (8,2)--(8,4);
\draw (12,2)--(12,4);
\draw (16,2)--(16,4);

\draw (0,4)--(2,6);
\draw (4,4)--(6,6);
\draw (8,4)--(10,6);
\draw (12,4)--(14,6);
\draw (16,4)--(18,6);

\draw (2,6)--(4,4);
\draw (6,6)--(8,4);
\draw (10,6)--(12,4);
\draw (14,6)--(16,4);

\draw (2,6)--(2,8);
\draw (6,6)--(6,8);
\draw (10,6)--(10,8);
\draw (14,6)--(14,8);
\draw (18,6)--(18,8);

\draw (2,8)--(4,10);
\draw (6,8)--(8,10);
\draw (10,8)--(12,10);
\draw (14,8)--(16,10);
\draw (18,8)--(20,10);

\draw (6,8)--(4,10);
\draw (10,8)--(8,10);
\draw (14,8)--(12,10);
\draw (18,8)--(16,10);

\draw (4,10)--(4,12);
\draw (8,10)--(8,12);
\draw (12,10)--(12,12);
\draw (16,10)--(16,12);
\draw (20,10)--(20,12);

\draw (4,12)--(6,14);
\draw (8,12)--(10,14);
\draw (12,12)--(14,14);
\draw (16,12)--(18,14);
\draw (20,12)--(22,14);

\draw (8,12)--(6,14);
\draw (12,12)--(10,14);
\draw (16,12)--(14,14);
\draw (20,12)--(18,14);

\draw (6,14)--(6,16);
\draw (10,14)--(10,16);
\draw (14,14)--(14,16);
\draw (18,14)--(18,16);
\draw (22,14)--(22,16);

\draw (6,16)--(8,18);
\draw (10,16)--(12,18);
\draw (14,16)--(16,18);
\draw (18,16)--(20,18);

\draw (10,16)--(8,18);
\draw (14,16)--(12,18);
\draw (18,16)--(16,18);
\draw (22,16)--(20,18);

\node (A1) at (4,4)[shape=circle, draw, fill=white] {};
\node (A2) at (8,4)[shape=circle, draw, fill=white] {};
\node (A3) at (12,4)[shape=circle, draw, fill=white] {};

\node (B1) at (6,6)[shape=circle, draw, fill=black] {};
\node (B2) at (10,6)[shape=circle, draw, fill=black] {};
\node (B3) at (14,6)[shape=circle, draw, fill=black] {};

\node (C1) at (6,8)[shape=circle, draw, fill=white] {};
\node (C2) at (10,8)[shape=circle, draw, fill=white] {};
\node (C3) at (14,8)[shape=circle, draw, fill=white] {};

\node (D1) at (8,10)[shape=circle, draw, fill=black] {};
\node (D2) at (12,10)[shape=circle, draw, fill=black] {};
\node (D3) at (16,10)[shape=circle, draw, fill=black] {};

\node (E1) at (8,12)[shape=circle, draw, fill=white] {};
\node (E2) at (12,12)[shape=circle, draw, fill=white] {};
\node (E3) at (16,12)[shape=circle, draw, fill=white] {};

\node (F1) at (10,14)[shape=circle, draw, fill=black] {};
\node (F2) at (14,14)[shape=circle, draw, fill=black] {};
\node (F3) at (18,14)[shape=circle, draw, fill=black] {};

\draw [dashed](2,3)--(8,15)--(20,15)--(14,3)--(2,3);
 \end{tikzpicture}

\end{center}
\caption{The Pappus graph and its embedding into torus.}
\label{PAPPUS}
\end{figure}

\begin{remark}
The idea of transpositional generating pairs appeared in \cite{JNS2007}, where the authors investigated a particular case called isobicyclic generating pairs. More specifically, a transpositional generating pair $(x,y)$ of a group $G$ is \textit{isobicyclic} if $G=\la x\ra\la y\ra$ and $\la x\ra\cap \la y\ra=1$. It was fundamental to observe that the isomorphism classes of regular embeddings $\M$ of $K_{n,n}$ with $\Aut_0^+(\M)\cong G$ are in one-to-one correspondence with orbits of isobicyclic generating pairs of $G$ under the action of $\Aut(G)$. Employing this correspondence, a complete classification of regular embeddings of  complete bipartite graphs was obtained in a series of papers \cite{NSZ2002, JNS2007, JNS2008, DJKNS2007, DJKNS2010, Jones2010}. The idea of isobicyclic generating pairs was also applied by Zhang and Du in the classification of regular embeddings of complete multipartite graphs~\cite{ZD2012,ZD2014}. Moreover, totally symmetric dessins appeared in \cite{Jones2013} as \textit{$\Omega$-invariant} regular dessins where $\Omega$ denotes the group of dessin operations.
\end{remark}

\section{Decomposition theorem}

It is well-known that a nilpotent group decomposes into a direct product
of its Sylow subgroups. In this section we show that an analogous statement
holds for nilpotent dessins. Before proceeding further we first summarise some prerequisites on nilpotent groups, for details the reader is referred to \cite[Chapter III]{Huppert1967}.

Let $G$ be a finite group. The \textit{upper central series} of $G$ is defined to be a series \[1=Z_0(G)\leq Z_1(G)\leq\cdots\leq Z_i(G)\leq\cdots  \] where
$Z_1(G)=Z(G)$ and $Z_{i+1}(G)/Z_i(G)=Z(G/Z_i)$, $i\geq1$. The \textit{lower central series} of $G$ is defined to be a series \[G=G_1\geq G_2\geq G_3\geq\cdots \] where $G_{i+1}=[G_i,G]$, $i\geq1$. A group $G$ is \textit{nilpotent} if the upper central series of $G$ contains $G$. If, in addition, the length of the upper central series of $G$ is $c$, then $G$ is said of (nilpotency) \textit{class} $c$; we write $c=c(G)$.

\begin{proposition}\cite[Chapter III]{Huppert1967}\label{NILP} Let $G$ be a finite group. The following statements hold true.
\begin{enumerate}
\item[\rm(i)]$G$ is nilpotent if and only if $G$ is the direct product of its Sylow subgroups.
\item[\rm(ii)]If $G$ is nilpotent, then its upper and lower central series have equal length.
\item[\rm(iii)]If $G$ is nilpotent of class $c$, then every subgroup of $G$ is nilpotent of class at most $c$.
\item[\rm(iv)] If $G=\la M\ra$ where $M\subseteq G$, then \[G_i=\la [x_1,x_2,\ldots, x_i]^g\mid x_i\in M, g\in G\ra=\la [x_1,x_2,\ldots, x_i], G_{i+1}\mid x_i\in M\ra.\]
\end{enumerate}
\end{proposition}

In the theory of regular dessins, the operation corresponding to direct product is the parallel product introduced by S. Wilson in \cite{Wilson1994}.  We first introduce it and then we show that each regular dessin with a nilpotent automorphism group can be decomposed into a parallel product of regular dessins whose automorphism groups are $p$-groups.
\begin{definition}\label{PARA}
Let $\D_i$ $(i=1,2)$ be regular dessins and $N_i$ the associated dessin subgroups. The \textit{parallel product} (or \textit{join}) $\D_1\Join\D_2$ of $\D_1$ and $\D_2$ is the regular dessin corresponding to the subgroup $N_1\cap N_2$ of $\Delta$.
\end{definition}
In the language of algebraic dessins, the parallel product of regular dessins can be translated as follows.
\begin{proposition}\cite{BN2001}
Let $\D_i=(G_i,x_i,y_i)$ $(i=1,2)$ be regular dessins. Then the parallel product $\D_1\Join\D_2$ of $\D_1$ and $\D_2$ is the regular dessin $(G,x,y)$ where $x=(x_1,x_2),$ $y=(y_1,y_2)$ and $G=\la x,y\ra$ is the subgroup generated by $x$ and $y$ in the direct product $G_1\times G_2$ of the groups $G_1$ and $G_2$.
\end{proposition}
The group $G$, being generating by $(x_1,x_2)$ and $(y_1,y_2)$, is  a subgroup of $G_1\times G_2$, possibly a proper one. It is proved that if $G_i$ $(i=1,2)$ are non-abelian simple subgroups, then $G=G_1\times G_2$~\cite[Corollary 6.2]{Jones2013}. The following proposition gives another sufficient condition for this.
\begin{proposition}\label{COPR}
Let $\D_i=(G_i,x_i,y_i)$ $(i=1,2)$ be regular dessins of types $(l_i,m_i,n_i)$, and let $\D=(G,x,y)$ be the parallel product of $\D_1$ and $\D_2$ where $x=(x_1,x_2),$ $y=(y_1,y_2)$ and $G=\la x,y\ra$. Let    any two of the following  conditions be satisfied:
\[\gcd(l_1,l_2)=1,\quad \gcd(m_1,m_2)=1,\quad \gcd(n_1,n_2)=1.\]
Then $G=G_1\times G_2$.
\end{proposition}
\begin{proof}Assume $\gcd(l_1,l_2)=1=\gcd(m_1,m_2)$. Since $G_1=\la x_1,y_1\ra$ and $G_2=\la x_2,y_2\ra$, we have
\[G_1\times G_2=\la (x_1,1),(y_1,1),(1,x_2),(1,y_2)\ra.\]
Clearly, $G=\la (x_1,x_2),(y_1,y_2)\ra\leq G_1\times G_2$. To see the opposite inclusion, it is sufficient to show that the generators of $G_1\times G_2$ are all contained in $G$. Since $o(x_1)=l_1$, $o(x_2)=l_2$ and $\gcd(l_1,l_2)=1$, we have
\[(1,x_2)\in\la (1,x_2^{l_1})\ra=\la(x_1,x_2)^{l_1} \ra\leq G\]
and
\[
 (x_1,1)\in \la (x_1^{l_2},1)\ra=\la (x_1,x_2)^{l_2}\ra\leq G.\]
Similarly, $(1,y_2)\in G$ and $(y_1,1)\in G$. Therefore, $G=G_1\times G_2$.

Set $z_i=x_iy_i$. Then $G_i=\la x_i,y_i\ra=\la x_i,z_i\ra=\la y_i,z_i\ra$. By symmetry, the other cases can be proved in a similar way.
\end{proof}

The following example shows that the condition in Proposition~\ref{COPR} is sufficient but not necessary.
\begin{example}
It is evident that the two reflexible symmetric dessins $\D_3$ and $\D_4$ in Example~\ref{CUBE}, being of type $(3,3,2)$ and $(3,3,3)$,  do not satisfy the  relative coprimeness conditions in Proposition~\ref{COPR}. It can be verified (by MAGMA, for instance) that the parallel product $\D=\D_3\Join\D_4$ is a reflexible symmetric dessin of type $(3,3,6)$ and genus $13$ with $|\Aut(\D)|=144$. It is the Conder's hypermap RPH13.2~\cite{Conder2012}. Since $\Aut(\D_3)\cong\Aut(\D_4)\cong\Alt_4$, $|\Aut(\D_3)\times \Aut(\D_4)|=144$. Therefore, $\Aut(\D)=\Aut(\D_3)\times \Aut(\D_4).$
\end{example}

\begin{proposition}\label{SRS}
Let $\D_i=(G_i,x_i,y_i)$ $(i=1,2)$ be regular dessins, and let $\D=\D_1\Join \D_2$. If $\gcd(|G_1|,|G_2|)=1$, then $\D$ possesses an external symmetry $\sigma$ if and only if so are $\D_1$ and $\D_2$.\end{proposition}

\begin{proof}Let $\D=(G,x,y)$ where $x=(x_1,x_2)$, $y=(y_1,y_2)$ and $G=\la x,y\ra$. Since $\gcd(|G_1|,|G_2|)=1$, by Proposition~\ref{COPR}, $G=G_1\times G_2$. It is well-known that each automorphism $\sigma\in\Aut(G)$ is induced by the automorphisms $\sigma_i\in\Aut(G_i)$ $(i=1,2)$ given by
\[\sigma(g_1,g_2)=(\sigma_1(g_1),\sigma_2(g_2)),\quad \text{for all $g_i\in G_i \,(i=1,2)$.}
\]
So $\Aut(G)\cong\Aut(G_1)\times\Aut(G_2)$ \cite[Theorem I.9.4]{Huppert1967}. It follows that the generating pair $(x,y)$ of $G$ possesses a $\sigma$-property if and only if so are the generating pairs of $G_i$ $(i=1,2)$, as required. \end{proof}

The following example shows that the condition $\gcd(|G_1|,|G_2|)=1$ of Proposition~\ref{SRS} is indispensable.
\begin{example}
Let $G_1$ and $G_2$ be the groups given as follows:
\begin{align*}&G_1=\la x_1,y_1\mid x_1^3=y_1^2=[x_1,y_1]=1\ra\cong \Z_3\oplus\Z_2,\\
&G_2=\la x_2,y_2\mid x_2^2=y_2^3=[x_2,y_2]=1\ra\cong\Z_2\oplus\Z_3.
\end{align*}
Let $G=\la x,y\ra$ where $x=(x_1,x_2)$ and $y=(y_1,y_2)$. Then
\[
G=\la x,y\mid x^6=y^6=[x,y]=1\ra\cong\Z_6\oplus\Z_6.
\]
It is easily seen that the dessin $\D=(G,x,y)$ is symmetric (in fact, totally symmetric), while the dessins $\D_i=(G_i,x_i,y_i)$ $(i=1,2)$ are not.

\end{example}


The following proposition reveals the relationship between the nilpotency class of a nilpotent regular dessin and the class of its shadow dessin.
\begin{proposition}\label{SM}
Let $\D=(G,x,y)$ be a nilpotent regular dessin, and let $\bar\D=(\bar G,\bar x,\bar y)$ be its shadow dessin. If $c(G)=c\geq 2$, then $c-1\leq c(\bar G)\leq c$. Conversely, if $c(\bar G)=c$, then $c\leq c(G)\leq c+1$.
\end{proposition}
\begin{proof}Let $K=\la x\ra\cap\la y\ra$. Assume $c(G)=c\geq 2$. Clearly, $c(\bar G)\leq c$. Note that $Z(G)<G$ and $c(G/Z(G))=c-1$. Since $K\leq Z(G)$, we have $G/Z(G)\cong (G/K)/(Z(G)/K)$. It follows that $c(\bar G)\geq c-1$. Conversely, assume $c(G/K)=c$. Then $c\leq c(G)$. To prove $c(G)\leq c+1$, suppose to the contrary that $c(G)\geq c+2$. Then the previous assertion implies that $c(G/K)\geq c+1$, a contradiction. \end{proof}

In order to simplify the explanation, a regular dessin  whose automorphism group is a $p$-group will be called a regular \textit{$p$-dessin}. The following  theorem allows us to reduce the classification of nilpotent regular dessins to the investigation of regular $p$-dessins.
\begin{theorem}[Decomposition Theorem]\label{REDUCTION}
Let $G\cong\prod_{i=1}^kG_i$ be the direct product decomposition of a finite two-generated nilpotent group $G$ where $G_i$ are the Sylow $p_i$-subgroups of $G$. Then a regular dessin $\D=(G,x,y)$ is the parallel product of regular dessins $\D_i=(G_i,x_i,y_i)$ where $x_i$ and $y_i$ are the images of $x$ and $y$ under the natural projections $\pi_i:G\to G_i$.

In particular, the dessin $\D$ possesses an external symmetry $\sigma$ if and only if so are the dessins $\D_i$ $(i=1,2,\ldots, k)$.
\end{theorem}
\begin{proof}Since $G$ is nilpotent, the natural projection $\pi_i:G\to G_i$ sends the generating pair $(x,y)$ of $G$ to the generating pair $(x_i,y_i)$ of the Sylow $p_i$-subgroups $G_i$. Therefore, $\D_i=(G_i,x_i,y_i)$ are well-defined regular dessins. The assignment $x\mapsto (x_1,x_2,\ldots, x_k), y\mapsto (y_1,y_2,\ldots, y_k)$ extends to an isomorphism from $G$ onto $\prod_{i=1}^kG_i$. Hence $\D$ is the parallel product of the dessins $\D_i$. Employing induction and  Proposition~\ref{SRS}, the dessin $\D$ possesses an external symmetry $\sigma$ if and only if so are the dessins $\D_i$ $(i=1,2,\ldots, k)$.
\end{proof}

\begin{remark}
 The parallel product has been mostly investigated in the category of maps. It was first defined by Wilson in \cite{Wilson1994}, and then investigated by Breda and Nedela under the name \textit{join} in \cite{BN2001} in terms of the lattice structure of the subgroups of the triangle groups. It was also called \textit{direct product} in \cite{MNS2012}, and \textit{Cartesian product} in \cite{JNS2008} in a slightly more restricted sense.
 \end{remark}

\section{Abelian regular dessins}
A regular dessin with an abelian automorphism group will be called an \textit{abelian regular dessin}.
In this section, abelian regular dessins are classified.  They were investigated
in \cite{Hidalgo2013} as well. However, we give a more detailed analysis
of such dessins including the enumeration of isomorphism classes and computation of main parameters.

We first study the cyclic regular dessins. The following preliminary results will be useful.
\begin{lemma}\cite{Hua1982}\label{CONGRUENCE}
Let $m$ and $n$ be positive integers, and let $a_i$ $(i=1,2,\cdots, n)$ be integers. Then the congruence $a_1x_1+a_2x_2+\cdots+a_nx_n\equiv b\pmod{m}$ is solvable if and only if $\gcd(a_1,a_2,\cdots a_n,m)|b$.
\end{lemma}
\begin{lemma}\cite{HNW2014}\label{LIFT}
Let $m$ and $n$ be positive integers where $m|n$. Then for each number $s$, $1\leq s< m$, such that $\gcd(s,m)=1$, there is a number $s'$, $1\leq s'< n$, such that $\gcd(s',n)=1$ and $s'\equiv s\pmod{m}$.
\end{lemma}

\begin{lemma}\cite{Hu2013}\label{CRIT}
Let $m$ be a positive integer, and let $a,b$ be integers. Then the congruence $ax\equiv b\pmod{m}$ has a solution $x\in\Z_m^*$ if and only if $\gcd(a,m)=\gcd(b,m)$.
\end{lemma}

\begin{theorem}\label{CYCLIC}
Let $C_m=\la g\ |\ g^m=1\ra$ be the cyclic group of order $m$ and  $\C(m;r,s)=(C_m,g^r,g^s) $ be the regular dessin corresponding to a generating pair $(g^r,g^s)$ of $C_m$. Then the following hold true:
\begin{enumerate}
\item[\rm(i)]a pair $(g^r,g^s)$ generates $C_m$ if and only if $\gcd(r,s,m)=1$, moreover, $\C(m;r,s)\cong\C(m;t,q)$ if and only if $t\equiv rk\pmod{m}$ and $q\equiv sk\pmod{m}$ for some $k\in\Z_m^*$;
\item[\rm(ii)]a pair $(g^r,g^s)$ is a transpositional generating pair of $C_m$ if and only if $\gcd(r,m)=1$ and $s\equiv re\pmod{m}$ where $e^2\equiv1\pmod{m}$, moreover, $\C(m;r, re)\cong\C(m;t,tf)$ if and only if $e\equiv f\pmod{m}$.
\end{enumerate}
\end{theorem}
\begin{proof}
(i) Since $C_m=\la g\ra$ is cyclic of order $m$, a pair $(g^r,g^s)$ generates $C_m$ if and only if there exist some integers $i$ and $j$ such that $g=g^{ri}g^{sj},$ that is, the equation $ri+sj\equiv1\pmod{m}$ is solvable. By Lemma~\ref{CONGRUENCE}, this is equivalent to the condition $\gcd(r,s,m)=1$. Moreover, by Theorem~\ref{CORR}, $\C(m;r,s)\cong\C(m;t,q)$ if and only if there is an automorphism of $C_m$ sending $(g^r,g^s)$ to $(g^t,g^q)$. Since the automorphisms of $C_m$ have form $g\mapsto g^k$ for some $k\in\Z_m^*$, we have $t\equiv rk\pmod{m}$ and $q\equiv sk\pmod{m}$.

(ii) If $(g^r,g^s)$ is a transpositional generating pair of $C_m$, then there is an automorphism $\tau:g\mapsto g^e$ of $C_m$ where $e\in\Z_m^*$ such that $\tau:g^r\mapsto g^s=g^{re}, g^s\mapsto g^r=g^{se}$. Hence, $s\equiv re\pmod{m}$ and $r\equiv se\pmod{m}$. It follows that $r(e^2-1)\equiv0\pmod{m}$. Recall that $C_m=\la g^r,g^s\ra$. Since $s\equiv re\pmod{m}$ and $\gcd(e,m)=1$, we have $\gcd(r,m)=1$. Hence, $e^2\equiv1\pmod{m}$. Conversely,  it is straightforward to verify that each pair $(g^r,g^{re})$ where $\gcd(r,m)=1$ and $e^2\equiv1\pmod{m}$ is a transpositional generating pair of $C_m$.

  Let  $r,t,e,f\in\Z_m^*$ where $e^2\equiv1\pmod{m}$ and $f^2\equiv1\pmod{m}$. By (i), $\C(m;r,re)\cong\C(m;t,tf)$ if and only if $t\equiv rk\pmod{m}$ and $tf\equiv rek$ for some $k\in\Z_m^*$. Since $r,t\in\Z_m^*$, we have $k\equiv r^{-1}t\pmod{m}$. Thus, the isomorphism conditions are reduced to $e\equiv f\pmod{m}$, as required.\end{proof}

\begin{corollary}Let $\C(m;r,s)$ denote the dessins from Theorem~\ref{CYCLIC}. Then under the generalised Wilson's operation, two such dessins $\C(m;r,s)$ and $\C(m;t,q)$ belong to the same orbit if and only if $\gcd(r,m)=\gcd(t,m)$ and $\gcd(s,m)=\gcd(q,m)$. In particular, every such dessin is invariant under the Wilson's operation.
\end{corollary}
\begin{proof}With the notation above, $\C(m;r,s)=(C_m,g^r,g^s)$. Recall that the generalised Wilson's operation $H_{i,j}$ sends the dessin $\C(m;r,s)$ to the dessin $\C(m;ri,sj)=(C_m,g^{ri},g^{sj})$ where $i\in\Z_{m/\gcd(r,m)}^*$ and $j\in\Z_{m/\gcd(s,m)}^*$. By Lemma~\ref{LIFT}, the elements $i$ and $j$ lift to the elements $i',j'\in\Z_m^*$ such that $i'\equiv i\pmod{m/\gcd(r,m)}$ and $j'\equiv j\pmod{m/\gcd(s,m)}$. Therefore, by Theorem~\ref{CYCLIC}, two such dessins $\C(m;r,s)$ and $\C(m;t,q)$ belong to the same orbit if and only if both the equations
\begin{align}\label{GWCOND}
t\equiv rx\pmod{m}\quad \text{and}\quad q\equiv sy\pmod{m}
\end{align}
have solutions $x=ki'$ and $y=kj'$ in $\Z_m^*$. By Lemma~\ref{CRIT}, this is equivalent to that $\gcd(r,m)=\gcd(t,m)$ and $\gcd(s,m)=\gcd(q,m)$.

In particular, since $H_j=H_{j,j}$, we deduce from the above discussion that two such dessins $\C(m;r,s)$ and $\C(m;t,q)$ belong to the same orbit under $H_j$ if and only if the equations in \eqref{GWCOND} have a common solution $x=y\in\Z_m^*$. By Theorem~\ref{CYCLIC}, this means that the dessins are isomorphic,   as required.
\end{proof}

In the following statement we summarise the main properties
of the cyclic dessins.

\begin{corollary} Let $\C(m;r,s)$ be the cyclic regular dessin from Theorem~\ref{CYCLIC}. Let $d_1=\gcd(m,r)$, $d_2=\gcd(m,s)$ and $d_3=\gcd(m,r+s)$.
Then
\begin{itemize}
\item[\rm(i)] $\C(m;r,s)$ is reflexible,
\item[\rm(ii)] the type of $\C(m;r,s)$ is $(m/d_1,m/d_2,m/d_3)$,
\item[\rm(iii)] the genus of $\C(m;r,s)$ is  $g=\frac{1}{2}\left(m+2-d_1-d_2-d_3\right)$,
\item[\rm(iv)]  the underlying graph of $\C(m;r,s)$ is the complete bipartite graph $K_{d_1,d_2}^{(c)}$ of multiplicity $c=m/\gcd(m,\lcm(r,s))$,
\item[\rm(v)] if $\C(m;r,s)$ is symmetric, then $d_1=d_2=1$ and the underlying graph is the $m$-dipole $K_{1,1}^{(m)}$.
\end{itemize}
\end{corollary}
\begin{proof}
For a generating pair $(g^r,g^s)$ of $C_m$, the assignment $g^r\mapsto g^{-r}, g^s\mapsto g^{-s}$ is an automorphism of $C_m$ induced by the automorphism $g\mapsto g^{-1}$ of $C_m$. Hence each dessin $\C(m;r,s)$ is reflexible.
The type of  $\C(m;r,s)$ is straithforward and the genus comes from Proposition~\ref{EPformula}.

To determine the underlying graph of $\C(m;r,s)$, let $u=\la g^r\ra$ be a black vertex. Then for any white vertex $v=g^l\la g^s\ra$, $0\leq l\leq s-1$, $u$ is adjacent to $v$ if and only if $\la g^r\ra\cap g^l\la g^s\ra\neq \varnothing$, or equivalently, there are integers $i,j$ such that  $g^{ir}=g^{l+js}$. This is equivalent to that the equation $ir-js\equiv l\pmod{m}$ is solvable. By Theorem~\ref{CYCLIC}(i), $\gcd(r,s,m)=1$. Hence, by Lemma~\ref{CONGRUENCE}, the equation $ir-js\equiv l\pmod{m}$ is solvable for each $l$, $0\leq l\leq s-1$, implying that the black vertex $u$ is adjacent to every white vertex. In particular, since $\la g^r\ra\cap\la g^s\ra=\la g^{\lcm(r,s)}\ra$, the dessin $\C(m;r,s)$ has multiplicity $c=m/\gcd(m,\lcm(r,s))$. It follows from the regularity that the underlying graph of $\C(m;r,s)$ is a complete bipartite graph $K_{d_1,d_2}^{(c)}$ with multiplicity $c$. In particular, if the dessin $\C(m;r,s)$ is symmetric, by Theorem~\ref{CYCLIC}(ii), $r\in\Z_m^*$ and $s=re$ for some $e\in\Z_m^*$. It follows that $d_1=\gcd(r,m)=1$ and $d_2=\gcd(s,m)=1$. Hence
\[c=m/\gcd(m,\lcm(r,s))=m/\gcd(m,\lcm(r,re))=m/\gcd(m,re)=m.\] Therefore, the underlying graph is the $m$-dipole, as claimed. \end{proof}

Now we turn to the classification of abelian regular dessins.  We first point out a folklore fact: Every complete bipartite graph $K_{n,m}$ underlies at least one (simple) regular dessin given by  $\D=(G, x,y)$ where
\[G=\la x,y\mid x^{n}=y^{m}=[x,y]=1\ra\cong\Z_{n}\oplus\Z_{m}.\]
 Following \cite{JNS2008},  we shall call it the \textit{standard embedding} of $K_{n,m}$.

 By the Decomposition Theorem \ref{REDUCTION}, the classification of abelian regular dessin is reduced to the classification of abelian regular $p$-dessins, as classified in the following theorem.

\begin{theorem}\label{ACLASS}
Let $a$ and $b$, $0\leq a\leq b$, be integers and $p$ be a prime. Then each regular dessin $\A=(G,x,y)$ where $G\cong\Z_{p^a}\oplus\Z_{p^b}$ and $o(x)=p^b$ is determined by the presentation
\begin{align}\label{APRE}
G=\la x,y\mid x^{p^{b}}=y^{p^{a+c}}=[x,y]=1, y^{p^a}=x^{ep^{b-c}}\ra,
\end{align}
where
 \begin{align}\label{(c,e)cond}
0\leq c\leq b-a\quad\text{and}\quad e\in\Z_{p^{c}}^*.
\end{align}
Moreover, up to  the duality swapping the the black and white vertices, the isomorphism classes of regular dessins $\A$ with $\Aut(\A)\cong \Z_{p^a}\oplus\Z_{p^b}$ are in one-to-one correspondence with the integer pairs $(c,e)$ satisfying {\rm(\ref{(c,e)cond})}.
\end{theorem}
\begin{proof}
Let $X=\la x\ra$ and $Y=\la y\ra$. Then $G=XY$. Note that $\exp(G)=p^b$. Since $G$ is abelian and $G=\la x,y\ra$, we have either $o(x)=p^b$ or $o(y)=p^b$. Up to the duality swapping the the black and white vertices, we may assume $o(x)=p^b$. Let $o(y)=p^n$, $n\leq b$. We have
\[p^{a+b}=|G|=|XY|=|X||Y|/|X\cap Y|=p^{b+n}/|X\cap Y|.\]
It follows that $n\geq a$. Assume $n=a+c$ where $0\leq c\leq b-a$. Then $|X\cap Y|=p^c$. Hence, there exists some integer $e\in\Z_{p^c}^*$ such that $y^{p^a}=x^{ep^{b-c}}$. Since $G$ is abelian, it has the presentation~\eqref{APRE}. Moreover, for fixed $a$, $b$ and prime $p$, it is evident that two such dessins corresponding to the pairs $(c_i,e_i)$ $(i=1,2)$ are isomorphic if and only if $c_1=c_2$ and $e_1\equiv e_2\pmod{p^{c_1}}$.
\end{proof}
Let $\A(p;a,b,c,e)$ denote the regular dessin from Theorem~\ref{ACLASS} corresponding to the pair $(c,e)$.
\begin{corollary}
For fixed nonnegative integers $a\leq b$ and prime $p$, under the generalised Wilson's operation, two regular abelian dessins $\A(p;a,b,c_i,e_i)$ $(i=1,2)$ belong to the same orbit if and only if $c_1=c_2$. In particular, each such dessin is invariant under the Wilson’s operation.

\end{corollary}
\begin{proof}
Recall that the generalised Wilson's operation $H_{i,j}$ replaces $(x,y)$ with $(x^i,y^j)$ where $i$ and $j$ are coprime to $p$. By the presentation \eqref{APRE}, we have $H_{i,j}(\A(p;a,b,c,e))\cong \A(p;a,b,c,i^{-1}je)$ where $i^{-1}$ is the modular inverse of $i$ in $\Z_{p^b}$. Therefore, under the operation $H_{i,j}$ two dessins $\A(p;a,b,c_i,e_i)$ $(i=1,2)$ belong to the same orbit if and only if $c_1=c_2$. Since $H_j=H_{j,j}$, under the Wilson's operation $H_j$ each such dessin is invariant.\end{proof}

The next corollary summarises some properties of the abelian $p$-dessins.
\begin{corollary} \label{ACLASS2}
Let $\A=\A(p;a,b,c,e)$ denote the regular $p$-dessin defined by the presentation~\eqref{APRE}, and let $d=\gcd(p^{b-a},1+ep^{b-a-c})$.
Then
\begin{itemize}
\item[\rm(i)] $\A$ is reflexible,
\item[\rm(ii)] the type of $\A$ is $(p^b,p^{a+c},p^{b}/d)$,
\item[\rm(iii)] the genus of $\A$ is  $\frac{1}{2}(1+p^a(p^b-p^{b-a-c}-d-1))$,
\item[\rm(iv)]  the underlying graph of $\A$ is the complete bipartite graph $K_{p^{b-c},p^a}^{(p^c)}$ of multiplicity $p^c$,
\item[\rm(v)] $\A$ is symmetric if and only if $c=b-a$ and $e^2\equiv1\pmod{p^c}$, in which case, its underlying graph is $K_{p^a,p^a}^{(p^{b-a})}$.
\end{itemize}
\end{corollary}

\begin{proof}
By replacing the generators $x$ and $y$ with  $x^{-1}$ and $y^{-1}$ we get an alternative presentation of $G$:
\begin{align}\label{APRE2}
G=\la x^{-1},y^{-1}\mid x^{-p^{b}}=y^{-p^{a+c}}=[x^{-1},y^{-1}]=1, y^{-p^a}=x^{-ep^{b-c}}\ra.
\end{align}
Comparing \eqref{APRE} with \eqref{APRE2} we see that $x\mapsto x^{-1}, y\mapsto y^{-1}$ extends to an automorphism of $G$. Hence the dessin $\A$ is reflexible. To determine the type, we need to calculate the order of $xy$. Since $y^{p^a}=x^{ep^{b-c}}$, we have \[(xy)^{p^a}=x^{p^a}y^{p^a}=x^{p^a}x^{ep^{b-c}}=x^{p^a(1+ep^{b-a-c})}.\] Hence, $o(xy)=p^b/\gcd(p^{b-a},1+ep^{b-a-c}).$ Therefore we obtain the type as stated in (ii). The genus follows from Proposition~\ref{EPformula}.

Further, it follows from the presentation~\eqref{APRE} that the shadow dessin of $\A$ is given by a disjoint abelian generating pair giving rise to the standard embedding of $K_{p^{b-c},p^a}$. Hence $\A$ has underlying graph isomorphic to $K_{p^{b-c},p^a}^{(p^c)}$.

Finally, if $\A$ is symmetric, then $o(x)=o(y)=p^b$. Hence, $c=b-a$. By Proposition~\ref{TPROPERTY}, $e^2\equiv1\pmod{p^{c}}$. Conversely, if the numerical conditions are satisfied, then $x\mapsto y, y\mapsto x$ extends to an automorphism of $G$, and hence the dessin $\A$ is symmetric. By (iv), the underlying graph is $K_{p^{a},p^a}^{(p^{b-a})}$.\end{proof}

The following corollary enumerates the number of isomorphism classes of regular abelian $p$-dessins.
\begin{corollary}\label{ANUM1}
Up to isomorphism, the number of regular dessins $\A$ with $\Aut(\A)\cong \Z_{p^a}\oplus\Z_{p^b}$ where $0\leq a\leq b$ is equal to $\psi(p^{b-a})$ where $\psi$ is the Dedekind's totient function.
\end{corollary}
\begin{proof}Let $G=\la x,y\ra\cong \Z_{p^a}\oplus\Z_{p^b}$ where $0\leq a\leq b$. By Theorem~\ref{ACLASS}, the number of regular dessins $\A=(G,x,y)$ with $o(x)=p^b$ is equal to $\sum_{c=0}^{b-a}\varphi(p^c)$. By the duality swapping the vertex colours, the number of regular dessins $\mathcal{B}=(G,x,y)$ with $o(y)=p^b$ is also equal to $\sum_{c=0}^{b-a}\varphi(p^c)$. Note that $\A\not\cong\mathcal{B}$ except when $o(x)=o(y)=p^b$, in which case, each dessin $\A$ corresponding to the pair $(b-a,e)$ is isomorphic to the dessin $\mathcal{B}$ corresponding to the pair $(b-a,f)$ where $ef\equiv 1\pmod{p^{b-a}}$. Therefore, by the Inclusion-Exclusion Principle, the total number of regular dessins whose automorphism groups are isomorphic to $\Z_{p^a}\oplus\Z_{p^b}$ is equal to $2\big(\sum_{c=0}^{b-a}\varphi(p^c)\big)-\varphi(p^{b-a})$. If $a=b$, then this is equal to $1=\psi(1)$. Otherwise, we have $b>a\geq 0$, and hence
\begin{align*}
2\sum_{c=0}^{b-a}\varphi(p^c)-\varphi(p^{b-a})&=2(1+\sum_{c=1}^{b-a}(p^c-p^{c-1}))-p^{b-a}+p^{b-a-1}\\
&=p^{b-a}+p^{b-a-1}=\psi(p^{b-a}),
\end{align*}
as claimed.
\end{proof}

The following theorem enumerates the abelian (including cyclic) dessins.
\begin{theorem}\label{ANUMBER}
\label{mainenum}
Let $1\leq n\leq m$ be integers such that $n|m$. Then
\begin{enumerate}
\item[\rm(i)]the number of non-isomorphic regular dessins $\A$ with $\Aut(\A)\cong \Z_{n}\oplus\Z_{m}$ is equal to $\psi(m/n)$, where $\psi$ is the Dedekind totient function,

\item[\rm(ii)]the number of symmetric dessins $\A$ with $\Aut(\A)\cong \Z_{n}\oplus\Z_{m}$ is equal to the number of solutions to the equation $e^2\equiv1\pmod{m/n}$ in $\Z_{m/n}$,

\item[\rm(iii)]the group $G\cong \Z_{n}\oplus\Z_{m}$ underlies a unique regular dessin if and only if $m=n$, in which case the dessin is totally symmetric.
\end{enumerate}
\end{theorem}
\begin{proof}
Let $n$ and $m$  have the prime decompositions
 $n=p_1^{a_1}\cdots p_{l}^{a_l}$ and $m=p_1^{b_1}\cdots p_l^{b_l}q_1^{c_1}\cdots q_t^{c_t}$ where $1\leq a_i\leq b_i$ for $1\leq i\leq l$ and $c_j\geq1$ for $1\leq j\leq t$.
By the Fundamental Theorem of Finite Abelian Groups,
\[\Z_n\oplus\Z_m\cong \bigoplus_{i=1}^l(\Z_{p_i^{a_i}}\oplus\Z_{p_i^{b_i}})\bigoplus_{j=1}^t\Z_{q_j^{c_j}}.
\]
By the Decomposition Theorem~\ref{REDUCTION} and Corollary~\ref{ANUM1}, the total number of non-isomorphic regular dessins is equal to the product $\prod_{i=1}^l\psi(p^{b_i-a_i})\prod_{j=1}^t\psi(q^{c_j})=\psi(m/n)$, as it was required.

Item (ii) follows from Corollary~\ref{ACLASS2}(v), and Item (iii) follows from the fact that $\psi(m/n)=1$ if and only if $m/n=1$.\end{proof}

\begin{example}The group $C_6=\la g|\ g^6=1\ra$ underlies $\psi(6)=12$ non-isomorphic regular dessins $\C_i$ $(1\leq i\leq 12)$ given below.
\[
\begin{array}{llllll}
\text{Dessin}&(g^r,g^s) & \text{Type} & \text{Genus} & \text{Graph} & \text{Symmetric}\\
\hline
\C_1 & (g^0,g^1) & (1,6,6) & 0  & K_{6,1} & \text{No}\\
\C_2 & (g^1,g^0) & (6,1,6) & 0  & K_{1,6} & \text{No}\\
\C_3 & (g^1,g^5) & (6,6,1) & 0  & K_{1,1}^{(6)} & \text{Yes}\\
\C_4 & (g^1,g^2) & (6,3,2) & 1  & K_{1,2}^{(3)} & \text{No}\\
\C_5 & (g^1,g^3) & (6,2,3) &  1 & K_{1,3}^{(2)} & \text{No}\\
\C_6 & (g^2,g^1) & (3,6,2) & 1  & K_{2,1}^{(3)} & \text{No}\\
\C_7 & (g^2,g^3) & (3,2,6) & 1  & K_{2,3} & \text{No}\\
\C_{8} & (g^3,g^1) & (2,6,3) & 1  & K_{3,1}^{(2)} & \text{No}\\
\C_{9} & (g^3,g^2) & (2,3,6) & 1  & K_{3,2} & \text{No}\\
\C_{10} & (g^1,g^1) & (6,6,3) &  2 & K_{1,1}^{(6)} & \text{Yes}\\
\C_{11} & (g^2,g^5) & (3,6,6) & 2  & K_{2,1}^{(3)} & \text{No}\\
\C_{12} & (g^1,g^4) & (6,3,6) & 2  & K_{1,2}^{(3)} & \text{No}\\
\hline
\end{array}
\]
 The dessins form 3 orbits under the duality operations:
\[
\{\C_1,\C_2,\C_3\},\{\C_4,\C_5,\C_6,\C_7,\C_8,\C_9\},\{\C_{10},\C_{11},\C_{12}\}.
\]
They form 9 orbits  under the generalised Wilson's operation:
\[
\{\C_1\},\{\C_2\}, \{\C_{3},\C_{10}\}, \{\C_4,\C_{12}\},\{\C_5\},\{\C_6,\C_{11}\},\{\C_7\},\{\C_8\},\{\C_9\}.
\]
\end{example}

\begin{remark} The function $\psi(m)=m\Pi_{p|m}(1+\frac{1}{p})$ counting the number of cyclic dessins
introduced above is the well-known  Dedekind totient function. It was introduced by Dedekind in connection with modular forms. It appears in many other contexts, for instance in the formula describing the generating function of the Riemann zeta function. Similar to the Euler totient function it can be generalized to higher degree functions by setting $\psi_k(n)=\frac{J_{2k}(n)}{J_k(n)}$, where $J_k(n)$ is the Jordan totient function. The formula
for the number of cyclic dessins appeared explicitly in \cite[Eample 3.1]{Jones2013} derived from the Hall's counting principle \cite{Hall1936}.
\end{remark}

\section{Nilpotent regular maps and nilpotent symmetric dessins}
In this section, we investigate the relationship between symmetric dessins and regular bipartite maps.

 Let $\D=(G,x,y)$ be a symmetric dessin of type $(m,m,n)$ with an automorphism $\tau$ of $G$ transposing $x$ and $y$. If $x\neq y$, then $\tau\neq1$. The \textit{extended automorphism group} $\Aut^+(\D)$ of $\D$ is the group $\la\Aut(\D),\tau\ra=\Aut(\D)\rtimes\la \tau\ra.$ On the other hand, if $x=y$, then $\tau=1$. In this special case, the extended automorphism group of $\D$ is defined to be
\[\la x,L\ |\ x^{m}=L^2=[x,L]=1\ra\cong\Z_m\oplus\Z_2,\] corresponding to a regular embedding of an $m$-dipole. It is clear that the extended automorphism group of a symmetric dessin corresponds to the orientation-preserving automorphism group of the associated regular bipartite map~\cite{JNS2008}.

Regular maps with nilpotent automorphism groups are studied in~\cite{MNS2012}. The authors showed that
\begin{proposition}\cite[Theorem 3.4]{MNS2012}\label{NILBIPART}
Every regular map $\M$ with a non-abelian nilpotent automorphism group is bipartite.
\end{proposition}
It follows that each non-abelian nilpotent regular map $\M$ gives rise to a nilpotent symmetric dessin $\D$. The following  result strengthens the above statement.
\begin{theorem}\label{CLASSREDUCE}
Every regular map $\M$ with a nilpotent automorphism group of class $c\geq2$ gives rise to a nilpotent symmetric dessin $\D$ with the automorphism group of class at most $c-1$.
\end{theorem}
\begin{proof}
Let $G=\Aut^+(\M)=\la R, L\ra$, $L^2=1$. By the assumption, $G$ is nilpotent of class $c$, $c\geq 2$. By Proposition~\ref{NILBIPART}, the map $\M$ is bipartite. Let $x=R$, $y=R^L$ and $H=\Aut_0^+(\M)$. Then $H=\la x,y\ra\cong\Aut(\D)$. Since $G$ has class $c$, the subgroup $H$ of $G$ has class at most $c$, that is, $H_{c+1}=1$. By Proposition~\ref{NILP}(iv), we have
\begin{align}\label{CLASSC}
H_c=\la [x_1,x_2,\cdots, x_c]\mid x_i\in\{x,y\}\ra.
\end{align}
 Note that
\begin{align*}
[x,y]=&[R, R^L]=R^{-1}LR^{-1}LRLRL=R^{-1}(LR^{-1}LR)R(R^{-1}LRL)\\
=&R^{-1}[R,L]^{-1}R[R,L]=[R,[R,L]]\in G_3.
\end{align*}
Hence, for any $x_i\in\{x,y\}\subseteq H$, we have
\begin{align*}
[x,y,x_3,x_4,\cdots, x_c]=[[x,y],x_3,\cdots, x_c]\in G_{c+1}=1.
\end{align*}
Since $[y,x]=[x,y]^{-1}$, we also have $[y,x,x_3,x_4,\cdots, x_c]\in G_{c+1}=1$. Therefore, by \eqref{CLASSC}, $H_{c}=1$. So $H\cong\Aut(\D)$ has class at most $c-1$.\end{proof}

As shown in \cite{MNS2012} the classification of nilpotent regular maps is reduced to the  classification
of regular maps whose automorphism groups are $2$-groups.
Regular maps $\M$ whose automorphism groups are $2$-groups of class 2 are classified  in  \cite[Theorem 5.4]{MNS2012}. The authors showed that either  $\M\cong\M_1(n)= (G_1(n),R,L)$ for some $b\geq 2$, or $\M\cong \M_2(n)=(G_2(n), R,L)$ for some $n\geq 1$, where
\begin{align*}
&G_1(n)=\la R,L\mid R^{2^n}=L^2=1, [R,L]=R^{2^{n-1}}\ra,\\
&G_2(n)=\la R,L\mid R^{2^n}=L^2=T^2=[R,T]=[L,T]=1, T:=[R,L]\ra.
\end{align*} By Theorem~\ref{NILBIPART}, the subgroups $\Aut_0^+(\M_i(n))$ $(i=1,2)$ are abelian. So they must be covered by Corollary~\ref{ACLASS2}(v). We identify them as follows.  Let $x=R$ and $y=R^L$. Using substitution, we have
\begin{align*}
&\Aut^+_0(\M_1(n))=\la x,y\mid x^{2^n}=[x,y]=1,y=x^{2^{n-1}+1}\ra,\\
&\Aut^+_0(\M_2(n))=\la x,y\mid x^{2^n}=[x,y]=1,y^2=x^2\ra.
\end{align*}
Hence $\M_1(n)\cong \mathcal{A}(2;0,n,n,2^{n-1}+1)$ and $\M_2(n)\cong\mathcal{A}(2;1,n,n-1,1)$. It follows that the classification
of regular maps with a nilpotent group of automorphisms of class two can be obtained from the classification of abelian dessins. Similar identification can be performed for regular maps $\M$ whose automorphism groups are $2$-groups of class 3 classified in \cite{BDLNS}; the interested reader is referred to \cite{Wang2014} for details.

Now we turn to the inverse of Theorem~\ref{CLASSREDUCE}. The following proposition shows that the extended automorphism group of a nilpotent symmetric dessin is not necessarily nilpotent. Recall that a finite group $G$ is \textit{supersolvable} if it has a \textit{cyclic invariant series}, that is, a series $1=N_0\leq N_1\leq N_2\leq\cdots \leq N_s=G$ of normal subgroups of $G$ such that $N_{i}/N_{i-1}$ $(i=1,2,\cdots s)$ are all cyclic.
\begin{proposition}\label{SUPER}
The extended automorphism group $G$ of a symmetric $p$-dessin is supersolvable. In particular, it is nilpotent if and only if either $p=2$, or $p>2$ and $G\cong\Z_{p^e}\oplus\Z_2$ for some $e\geq0$.
\end{proposition}

\begin{proof} Let $\D=(H,x,y)$ be a symmetric $p$-dessin with $\tau\in\Aut(H)$ transposing $x$ and $y$. Assume $|H|=p^e$. If $\tau\neq1$, then the extended automorphism group $G=\Aut^+(\D)=H\rtimes\la \tau\ra$ has order $2p^e$. Since every group of order $2p^e$ is supersolvable \cite[Theorem 7.2.15]{Scott1964}, $G$ is supersolvable. On the other hand, if $\tau=1$, then by the preceding convention, $G\cong\Z_{p^e}\oplus\Z_2$, which is  supersolvable.

Moreover, assume $G$ is nilpotent and $p>2$, then the Sylow $2$-group generated by the (unique) involution $L$ of $G$ induced by the automorphism $\tau$ of $H$ transposing $x$ and $y$ is a direct factor of $G$. It follows that $\tau=1$ and $x=y$. Hence, $G\cong\Z_{p^e}\oplus\Z_2$, as required.
\end{proof}

Finally, an interesting result proved in \cite{CDNS} shows that the number of vertices in a simple nilpotent regular map is bounded, provided that the class of its automorphism group is bounded. The following example shows that similar statement does not extend to nilpotent dessins. In particular, it confirms the statement on nilpotent regular maps.
\begin{example}
Let $\D$ be the simple abelian symmetric $2$-dessin $(H,x,y)$ where
\[
H=\la x,y\mid x^{2^a}=y^{2^a}=[x,y]=1\ra, \quad a\geq1.
\]
Clearly, $c(H)=1$. The number of vertices of $\D$ is equal to $2^{a+1}$, which is unbounded.

Now let $\M$ be the regular map corresponding to $\D$. Then $\Aut^+(\M)$ is isomorphic to the extended automorphism group $G=\Aut^+(\D)$ of $\D$ which is defined by the presentation
\begin{align*}
G=\la x,\tau\mid x^{2^a}=\tau^2=[x,y]=1, y:=x^\tau\ra.
\end{align*}
We show that $c(G)=a+1$. From the presentation of $G$ we see that any element of $G$ can be written as form $x^iy^i\tau^l$ where $1\leq i,j\leq 2^a$ and $0\leq l\leq 1$. It follows that $w=x^iy^j\tau^l\in Z(G)$ if and only if $w^x=w$, $w^y=w$ and $w^\tau=w$.
If $l=1$, then $w^x=(x^iy^j\tau)^x=x^{i-1}y^{j+1}\tau\neq w$. Moreover, if $l=0$, then $w^\tau=(x^iy^j)^\tau=y^ix^j$. So $w\in Z(G)$ if and only if $i=j$ and $l=0$. Hence, $Z(G)=\la xy\ra\cong\Z_{2^a}$. It follows that the group $\bar G=G/Z(G)=\la \bar x,\bar\tau\ra$  has a presentation
\begin{align*}
\bar G=\la \bar x,\bar \tau\mid \bar x^{2^a}=\bar \tau^2=1,\bar x^{\bar \tau}=\bar x^{-1}\ra\cong D_{2^{a+1}}.
\end{align*}
It is well-known that $D_{2^{a+1}}$ is $2$-group of class $a$. Hence, $c(G)=a+1$. Therefore, if $c(G)$ is bounded, so is the number $2^{a+1}$ of vertices in the regular map $\M$.
\end{example}

\section*{Acknowlegement}
The main material of the paper is based on the third author's PhD thesis~\cite{Wang2014}. She is grateful to the support during last years from the Department of Mathematics in the Faculty of Natural Sciences of Matej Bel University in Bansk\'a Bystrica.
The authors are grateful to Prof. Shao-Fei Du for his generosity to share his unpublished work in \cite{BDLNS}, and to Prof. Gareth Jones for his illuminating suggestions which have improved the presentation of the paper. This work is supported by the following grants: APVV-0223-10, and the grant APVV-ESF-EC-0009-10 within the EUROCORES Programme EUROGIGA (Project GReGAS) of the European Science Foundation and the Slovak-Chinese bilateral grant APVV-SK-CN-0009-12.


\end{document}